\documentclass[12pt,reqno]{amsart}
\usepackage{fullpage}
\usepackage[top=1.5in, bottom=2in, left=2.5in, right=2.5in]{geometry}
\usepackage{color, kotex}
\usepackage{graphicx}
\usepackage[draft]{hyperref}
\usepackage{amsmath,amsopn,amssymb,amsfonts,stmaryrd}
\usepackage{verbatim}
\usepackage{amsthm}
\usepackage{mathtools}
\usepackage{color}
\usepackage{enumitem}
\usepackage[framemethod=TikZ]{mdframed}
\usepackage{bbm}
\usepackage{mathrsfs}
\usepackage{booktabs}
\usepackage{caption}
\usepackage{bm}
\usepackage{tensor}
\usepackage{cleveref}
\usepackage{tikz,pgfplots}
\usetikzlibrary{positioning}

\newcommand{\IZ}{{\mathbb Z}}
\renewcommand{\H}{\mathbb{H}}

\newcommand{\SL}{\mathrm{SL}}

\newcommand{\N}{\mathbb N}

\theoremstyle{plain}
\newtheorem{thm}{Theorem}[section]
\newtheorem*{thm*}{Theorem}

\newtheorem{lem}[thm]{Lemma}

\theoremstyle{definition}

\newtheorem{defn}[thm]{Definition}

\newtheorem*{lem*}{Lemma}

\newtheorem*{rem}{Remark}

\numberwithin{equation}{section}

\theoremstyle{remark}

\def\n{\nu}

\def\t{\tau}

\def\n{\nu}

\def\t{\tau}

\newcommand{\R}{\mathbb R}
\newcommand{\Z}{\mathbb Z}

\newcommand{\lrb}[1]{\left(#1\right)}

\setlist[itemize]{noitemsep, topsep=0pt}

\allowdisplaybreaks

\makeatletter
\newcommand{\vast}{\bBigg@{3}}
\newcommand{\Vast}{\bBigg@{5}}
\makeatother

\title{Traces of partition Eisenstein series and almost holomorphic modular forms}
\author{Kathrin Bringmann}
\author[B. Pandey]{Badri Vishal Pandey}
\address{University of Cologne, Department of Mathematics and Computer Science,
	Weyertal 86-90,
	50931 Cologne, Germany}
\email{kbringma@math.uni-koeln.de}
\email{bpandey@math.uni-koeln.de}

\makeatletter
\@namedef{subjclassname@2020}{\textup{2020} Mathematics Subject Classification}
\makeatother

\begin{document}
	\maketitle
	\begin{abstract}
		Recently, Amderberhan, Griffin, Ono, and Singh started the study of ``traces of partition Eisenstein series'' and used it to give explicit formulas for many interesting functions. In this note we determine the precise spaces in which they lie, find modular completions, and show how they are related via operators.
	\end{abstract}
	\section{Introduction and Statement of results}
	In \cite{AGOS} Amderberhan--Griffin--Ono--Singh introduced the functions 
	\[e_k\left(\Lambda_\tau(s)\right):=\sum_{\substack{\omega_1,\ldots,\omega_k\in\Lambda_\tau(s)\\ \text{distinct}}}\frac{1}{\omega_1\cdots\omega_k},\]
	for $k\in\N$ and $\t\in\H$, where
	\[
		\Lambda_\tau(s):=\left\{\omega^2|\omega|^{2s}:\omega\in\Lambda_\tau\setminus\{0\}\right\}.
	\]
	Here $\Lambda_\t := \Z+\Z\t$. Moreover, let 
	\begin{equation*}
		e_k(\Lambda_\tau(0)) := \lim_{s\to0^+} e_k(\Lambda_\tau(s)).
	\end{equation*}
	In Theorem 1.1 (1) of \cite{AGOS} it was shown that $e_k(\Lambda_\tau(0))$ is a weight $2k$ non-holomorphic modular form. In this note we make this statement precise and determine the space in which these lie and how they can be related via operators. Recall that the {\it lowering operator} is defined as, $\tau=u+iv$ throughout, $L := -2iv^2\frac\partial{\partial\overline\tau}$. An \emph{almost holomorphic modular form $\widehat f$} is a function on $\H$ that transforms like a modular form and is a polynomial in $\frac{1}{v}$ with holomorphic coefficients\footnote{See \cite[Subsection~5.3]{BGHZ2008} for an exposition on almost holomorphic modular forms.}. The degree of the polynomial is called the \emph{depth} of $\widehat{f}$. The holomorphic part $f$ of $\widehat f$ is called a \emph{quasimodular form}. Note that the quasimodular form can be obtained from the almost holomorphic modular form by\footnote{Here $\tau$ and $\overline\tau$ are considered as independent variables.}
	\begin{equation*}
		f(\tau) = \lim_{\overline{\t}\to-i\infty} \widehat f(\tau).
	\end{equation*}
	\begin{thm}\label{main}
		Let $k\in\N$.
		\begin{enumerate}[leftmargin=*,label=\rm(\arabic*)]
			\item We have that $e_k(\Lambda_\tau(0))$ is an almost holomorphic modular form of weight $2k$ and depth $k$.
			\item We have $e_0(\Lambda_\tau(0))=1$ and
			\[L\left(e_k(\Lambda_\tau(0))\right)=\frac{\pi}{2}e_{k-1}\left(\Lambda_\tau(0)\right).\]
			\item The (modified) generating function
			\begin{equation*}
				\phi(z;\tau) := e^{-\frac{\pi z^2}{2v}} \sum_{k\ge0} (-1)^k e_k(\Lambda_\tau(0)) z^{2k+1}
			\end{equation*}
			is a Jacobi form of weight $-1$ and index $\frac{1}{2}$.
		\end{enumerate}
	\end{thm}
	Next we look at the Tylor coefficients of a meromorphic Jacobi form (see Subsection~\ref{subsec:Jacobi-form} for the definition) with torsion divisors. Let $\phi(z;\t)$ be a meromorphic Jacobi form. The {\it divisor} of $\phi$ (with respect to $\t$) is the formal sum
	\begin{align*}
		\mathrm{Div}_\t(\phi):=\sum_{w\in \mathbb{C}\slash\Lambda_\t} a_w\cdot [w],
	\end{align*}
	where $a_w$ is the order of $\phi_\t(z):=\phi(z;\t)$ at the point $w$ and $[w]$ is treated as symbol representing the point $w$ (see \cite[Chapter~II.3]{S1986} for more details on divisors). Note that this is a finite sum. A point $z=a\t+b\in \mathbb{C}\slash\Lambda_\t$ is a {\it torsion point} if $a$, $b\in\mathbb Q$. We define the \emph{Eisenstein-theta function for the divisor} $D=\mathrm{Div}_\t(\phi)$ as
	\begin{align*}
		G_{k,D}(\t):=\sum_{w\in \mathbb{C}\slash \Lambda_\t} a_w G_{k,w}(\t).
	\end{align*}
	Here for a non-zero torsion point $w$
	\begin{align*}
		G_{k,w}(\t):=-\left[\lrb{\frac{1}{2\pi i}\frac{\partial}{\partial z}}^k \log\left(\left|\vartheta(z;\t)\right|^2\right) \right]_{z=w},
	\end{align*}
	where the theta function $\vartheta$ is defined in equation \eqref{E:JacobiTheta}.
	For convention set $G_{2k,0}(\t):=G_{2k}(\t)$ (see \eqref{eq:G-2k} for the definition of $G_{2k}(\t)$) and $G_{2k+1,0}(\t):=0$.	For a partition\footnote{Here and throughout we use $\lambda\vdash n$ for $(1^{m_1},2^{m_2},\cdots,n^{m_n})\vdash n$.} $\lambda\vdash n$, define the {\it $n$-th partition Eisenstein-theta function for the divisor $D$} as
	\begin{align*}
		\mathcal{G}_{D,\lambda}(\t):=G_{1,D}^{m_1}(\t)G_{2,D}^{m_2}(\t)\cdots G_{n,D}^{m_n}(\t).
	\end{align*}
	The {\it $n$-th partition Eisenstein-theta trace for $D$ and} $\psi:\mathcal{P}\to \mathbb{C}$, a function on partition, is defined as
	\begin{align}\label{eq:trace}
		\mathrm{Tr}_n(D,\psi;\t) := \sum_{\lambda\vdash n} \psi(\lambda) \mathcal{G}_{D,\lambda}(\t).
	\end{align}
	In Theorem~1.4 of \cite{AGOS}, it was shown that a meromorphic Jacobi form with torsion divisors $D$ can be expressed in terms of these partition Eisenstein-theta traces for $D$. To state this result for a partition $\lambda=(1^{m_1},2^{m_2},\ldots,n^{m_n})\vdash n$ we let
	\begin{align*}
		\psi_J(\lambda):=\frac{(-1)^{\ell(\lambda)}}{\prod_{j=1}^n m_j!j!^{m_j}},
	\end{align*}
	where $\ell(\lambda):=\sum_{j=1}^n m_j$ is the \emph{length} of $\lambda$.
	\begin{thm}\label{second}
		{\rm(\cite[Theorem 1.4]{AGOS})} Suppose that $\phi(z;\t)$ is a meromorphic Jacobi form of weight $k$ and index $m$, and torsion divisor $D$, and let $a$ be the order of $\phi(z;\t)$ at $z=0$. Then there exists a unique weakly holomorphic modular form $f_\phi$ of weight\footnote{Note that there is a typo in Theorem 1.4 of \cite{AGOS}: the weight there is $k-a$ instead of $k+a$.} $k + a$ for which, in a neighborhood around $0$,
		\begin{align*}
			\phi(z;\t)=f_\phi(\t)\sum_{n\ge 0} \mathrm{Tr}_n(D,\psi_J;\t)(2\pi i z)^{n+a}.
		\end{align*}
	\end{thm}
	We next determine the space in which $\mathrm{Tr}_n(D,\psi;\t)$ lie.
	For this, for $n\in\N$, we define (using the notation from \Cref{second})
	\begin{align}\label{eq:Tr-hat}
		\widehat{\mathrm{Tr}}_{n}(D,\psi_J;\t):= \sum_{0\le j\le \frac n2} \frac{\left(\frac{\pi m}v\right)^j}{j!} (2\pi i)^{n-2j+a} \operatorname{Tr}_{n-2j} (D,\psi_J;\tau).
	\end{align}
	\begin{thm}\label{Jacobi}\hspace{0cm}
		\begin{enumerate}[leftmargin=*,label=\rm(\arabic*)]
			\item The function $\widehat{\mathrm{Tr}}_{n}(D,\psi_J;\t)$ is a weight $n$ and depth at most $\lfloor\frac{n}{2}\rfloor$ almost holomorphic modular form with
			\begin{align*}
				\lim_{\overline{\t}\to-i\infty} \widehat{\mathrm{Tr}}_{n}(D,\psi_J;\t) = (2\pi i)^{n+a} \mathrm{Tr}_{n}(D,\psi_J;\t).
			\end{align*}
			\item We have 
			\begin{align*}
				\widehat{\mathrm{Tr}}_0(D,\psi_J;\t)& =  (2\pi i)^a, \quad \widehat{\mathrm{Tr}}_1(D,\psi_J;\t) = -(2\pi i)^{1+a}G_{1, D}(\t),\\
				L\lrb{\widehat{\mathrm{Tr}}_{n}(D,\psi_J;\t)} &= -\pi m \widehat{\mathrm{Tr}}_{n-2}(D,\psi_J;\t).
			\end{align*}
		\end{enumerate}
	\end{thm}
	Next, we study the following two sequences of $q$-series considered by Ramanujan in his lost notebook \cite[pp~369]{R1998} $(n\in\N)$
\begin{align*}
	U_{2n}(q) &= \frac{1^{2n+1} - 3^{2n+1}q + 5^{2n+1}q^3 - 7^{2n+1}q^6 + 9^{2n+1}q^{10} -\cdots}{1-3q+5q^3-7q^6+9q^{10}-\cdots}, \\
	V_{2n}(q) &= \frac{1^{2n}-5^{2n}q-7^{2n}q^2+11^{2n}q^5+13^{2n}q^7-\cdots}{1-q-q^2+q^5+q^7-\cdots}.
\end{align*}
Ramanujan offered many identities such as ($q :=e^{2\pi i\t}$)
\begin{align*}
	U_0(q)&=1,\quad U_{2}(q)=E_2(\t),\quad U_4(q)=\frac{1}{3}\left(5E_2^2(\t)-2E_4(\t)\right),\\ U_6(q)&=\frac{1}{9}\left(35E_2^3(\t)-42E_2(\t)E_4(\t)+16E_6(\t)\right), \cdots,\\
	\intertext{and}
	V_0(q)&=1,\quad V_{2}(q)=E_2(\t),\quad V_4(q)=3E_2^2(\t)-2E_4(\t),\\
	V_6(q)&=15E_2^3(\t)-30E_2(\t)E_4(\t)+16E_6(\t),\cdots ,
\end{align*}
where $E_{2k}$ are the normalized Eisenstein series defined in \eqref{eq:Ek}. Ramanujan further claimed that for $n\in\N$, $U_{2n}$ and $V_{2n}$ can be written in terms of $E_2, E_4,$ and $E_6$. Berndt, Chan, Liu, Yee, and Yesilyurt \cite{BCLY2004,BY2003} proved this claim. However, their result was not explicit.
This was recently made explicit in Theorem 1.2 of \cite{AOS2024} by writing $U_{2n}$ and $V_{2n}$ as partition Eisenstein trace. Define $\psi_1,\psi_2:\mathcal{P}\to \R$, for a partition $\lambda\vdash n$ by
\begin{align*}
	\psi_1(\lambda)&:=4^n(2n+1)!\prod_{k=1}^n\frac{1}{m_k!}\lrb{\frac{B_{2k}}{2k(2k)!}}^{m_k}\!, \\	\psi_2(\lambda)&:=4^n(2n)!\prod_{k=1}^n\frac{1}{m_k!}\lrb{\frac{\left(4^k-1\right)B_{2k}}{2k(2k)!}}^{m_k}\!.
\end{align*}
In \cite[Theorem~1.2]{AOS2024} it was shown that
\begin{align}\label{eq:U-V-Tr}
	U_{2n}(q)=\mathrm{Tr}_n(\psi_1;\t),\quad V_{2n}(q)=\mathrm{Tr}_n(\psi_2;\t),
\end{align}
	where, for any $\psi:\mathcal{P}\to \mathbb{C}$, we define
	\begin{align*}
		\mathrm{Tr}_n(\psi;\t) := \sum_{\lambda\vdash n} \psi(\lambda) E_{2}^{m_1}(\t)E_4^{m_2}(\t)\cdots E_{2n}^{m_n}(\t).
	\end{align*}
\begin{rem}
	The partition Eisenstein series considered here should not be confused with the beautiful functions of the same name studied by Just and Schneider \cite{JS2021}. In their work the partition Eisenstein series are defined as lattice sums that are dictated by partitions.
\end{rem}
As our next result, we determine the space in which $U_{2n}$ and $V_{2n}$ lie. Define
\begin{align}
	\widehat{U}_{2n}(\tau)&:=2i\sum_{0\le j\le n} (-1)^{n+j} \frac{\pi^{2n-j+1} U_{2n-2j}(q)}{(2v)^j j!(2n-2j+1)!},\label{eq:Uhat}\\
	\widehat{V}_{2n}(\t)&:=\sum_{0\le j\le n} (-1)^{n+j}\frac{ 3^j \pi^{2n-j} V_{2n-2j}(q)}{(2v)^j j!(2n-2j)!}.\label{eq:V}
\end{align}
	\begin{thm}\label{Ramanujan-U-V}\hspace{0cm}

		\begin{enumerate}[leftmargin=*,label=\rm(\arabic*)]
			\item The functions $\widehat{U}_{2n}$ and $\widehat{V}_{2n}$ are almost holomorphic modular forms of weight $2n$ and depth $n$.
			\item We have $\widehat{U}_0=2\pi i$, $\widehat{V}_0=1$, and
			\begin{align*}
				L\lrb{\widehat{U}_{2n}}=-\frac{\pi}{2}\widehat{U}_{2n-2},\quad L\lrb{\widehat{V}_{2n}}=-\frac{3\pi}{2}\widehat{V}_{2n-2}.
			\end{align*}
			\item We have 
			\begin{equation*}
				\hspace{0.6cm}\lim_{\overline{\t}\to-i\infty} \widehat U_{2n}(\tau) = \frac{ 2(\pi i)^{2n+1}U_{2n}(q)}{(2n+1)!}, \, \lim_{\overline{\t}\to-i\infty} \widehat V_{2n}(\tau) = \frac{(\pi i)^{2n}}{(2n)!} V_{2n}(q).
			\end{equation*}
		\end{enumerate}
	\end{thm}
	The paper is organized as follows. In Section~\ref{sec-pre}, we recall some basic facts about Eisenstein series, Jacobi forms, and P\'{o}lya cycle index polynomials. In Sections~\ref{sec-thm-1}, \ref{sec-thm-3}, and \ref{sec-thm-4}, we prove Theorems \ref{main}, \ref{Jacobi}, and \ref{Ramanujan-U-V}, respectively.

	\section*{Acknowledgements}
	The authors thank Ken Ono for helpful discussions. The authors have received funding from the European Research Council (ERC) under the European Union's Horizon 2020 research and innovation programme (grant agreement No. 101001179).
	\section{Preliminaries}\label{sec-pre}
	\subsection{Eisenstein series}
	The \emph{Eisenstein series of weight $k$} are defined as
	\begin{align}\label{eq:G-2k}
	G_{2k}(\tau):=-\frac{B_{2k}}{2k}+2\sum_{n\ge1}\sigma_{2k-1}(n)q^n,
	\end{align}
	where $B_\ell$ denotes the $\ell$-th Bernoulli number and $\sigma_\ell(n):=\sum_{d\mid n}d^\ell$ the \emph{$\ell$-th divisor sum}. For $k>1$ the Eisenstein series are modular forms whereas $G_2$ can be completed by adding a simple term. That is defining
	\begin{equation}\label{Gh}
		\widehat{G}_2(\tau):=G_2(\tau)+\frac{1}{4\pi v}
	\end{equation}
	we have for $
	\begin{psmallmatrix}
		a & b\\
		c & d
	\end{psmallmatrix}
	\in\SL_2(\IZ)$, for $k\ge 2$,
	\begin{equation}\label{E:Trans}
		G_{2k}\left(\frac{a\tau+b}{c\tau+d}\right) = (c\tau+d)^{2k} G_{2k}(\tau), \quad \widehat{G}_2\left(\frac{a\tau+b}{c\tau+d}\right) = (c\tau+d)^2 \widehat{G}_2(\tau).
	\end{equation}
	We also require the {\it normalized Eisenstein series}
	\begin{equation}\label{eq:Ek}
		E_{2k}(\tau) := -\frac{2k}{B_{2k}}G_{2k}(\tau).
	\end{equation}
	\subsection{Jacobi forms}\label{subsec:Jacobi-form}
	\begin{defn}
		A {\it holomorphic Jacobi form of weight $k$ and index $m$} ($k\in \N, m \in \frac{1}{2}\mathbb N$) is a holomorphic function
		$\phi:\mathbb C \times \mathbb H \to
		\mathbb C$ which, for all $\left(\begin{smallmatrix}a&b\\c&d\end{smallmatrix}\right) \in \SL_2(\Z)$ and $\lambda,\mu \in \mathbb Z$, satisfies ($\zeta:=e^{2\pi iz}$):
		\begin{enumerate}[leftmargin=*]
			\item $\phi\left(\frac{z}{c\tau + d};\frac{a\tau+b}{c\tau+d}\right) = (c\tau + d)^k e^{\frac{2\pi i m c z^2}{c\tau + d}} \phi(z;\tau)$,
			\item $\phi(z + \lambda \tau + \mu;\tau) = (-1)^{2m\mu+\lambda \varepsilon}e^{-2\pi i m (\lambda^2\tau + 2\lambda z)} \phi(z;\tau)$, for some $\varepsilon\in\{0,1\}$.
			\item The function $\phi$ has a Fourier expansion $\phi(z;\tau)=\sum_{n\ge 0, r\in \Z} c(n,r)q^n \zeta^r$ with $c(n,r)=0$ unless $r^2\le 4nm$.
		\end{enumerate}
		Moreover, we call a Jacobi form $\phi$ {\it meromorphic} if together with (1) and (2), we allow it to have poles in $z$, and where the expansions in (3) allow for finite order poles in $\t$ at cusps. In this case $k$ can also be non-positive.
	\end{defn}
	The {\it Jacobi theta function} is a Jacobi form of weight $\frac12$ and index $\frac12$ (under a slightly modified definition)
	\begin{equation}\label{E:JacobiTheta}
		\vartheta(z;\tau) := \sum_{n\in\Z+\frac12} e^{2\pi in\left(z+\frac12\right)} q^\frac{n^2}2.
	\end{equation}
	To be more precise, it satisfies the elliptic transformation ($\lambda$, $\mu\in\Z$)
	\begin{equation}\label{E:TEl}
		\vartheta(z+\lambda\tau+\mu;\tau) = (-1)^{\lambda+\mu} q^{-\frac{\lambda^2}2} \zeta^{-\lambda} \vartheta(z;\tau)
	\end{equation}
	and the modular transformation ($
	\begin{psmallmatrix}
		a&b\\
		c&d
	\end{psmallmatrix}
	\in\SL_2(\Z)$)
	\begin{equation}\label{E:TMod}
		\vartheta\left(\frac z{c\tau+b};\frac{a\tau+b}{c\tau+d}\right) = \nu_\eta^3
		\begin{pmatrix}
			a&b\\
			c&d
		\end{pmatrix}
		(c\tau+d)^\frac12 e^\frac{\pi icz^2}{c\tau+d} \vartheta(z;\tau).
	\end{equation}
	Here $\nu_\eta
	\begin{psmallmatrix}
		a&b\\
		c&d
	\end{psmallmatrix}
	$ is the multiplier of the {\it Dedekind $\eta$-function}
	\begin{equation*}
		\eta(\t):=q^{\frac{1}{24}}\prod_{n\ge 1}(1-q^n),
	\end{equation*}
	i.e.,
	\begin{equation}\label{eq:eta}
		\eta\left(\frac{a\tau+b}{c\tau+d}\right) = \nu_\eta
		\begin{pmatrix}
			a&b\\
			c&d
		\end{pmatrix}
		(c\tau+d)^\frac12 \eta(\tau).
	\end{equation}
	Note that $z\mapsto\vartheta(z;\tau)$ has roots precisely for $z\in\IZ+\IZ\tau$ and these are simple.
	\subsection{P\'{o}lya cycle index polynomials}
	We require a result for P\'{o}lya cycle index polynomials in the case of symmetric groups that is required to prove our results; see \cite{S1999} for more details on cycle index polynomials. Let $S_n$ be the symmetric group; the set of permutations of the symbols $x_1,x_2,\ldots,x_n$. The \emph{cycle index polynomial} for $S_n$ is given by
	\begin{align*}
		Z(S_n):= \sum_{\lambda\vdash n} \prod_{k=1}^n\frac{1}{m_k!} \lrb{\frac{x_k}{k}}^{m_k}.
	\end{align*}
	\begin{lem}[Example 5.2.10 of \cite{S1999}]\label{lem:cycle-index}
		We have
		\begin{align*}
			\sum_{n\ge 0} Z(S_n) w^n = \exp\lrb{\sum_{k\ge 1} x_k \frac{w^k}{k}},
		\end{align*}
		as a formal power series in $w$.
	\end{lem}
	\section{Proof of Theorem \ref{main}}\label{sec-thm-1}
	In this section we prove Theorem 1.1.
	\begin{proof}[Proof of Theorem 1.1]\hspace{0cm}

		\noindent(1) The modular transformation property follows directly from Theorem~1.1 (1) of \cite{AGOS}.  For the readers convenience however we give a complete proof. We have
		\begin{equation*}
			e_k(\Lambda_\tau(0)) = \lim_{s\to0^+} e_k(\Lambda_\tau(s)).
		\end{equation*}
		Consider the generating function (see e.g. (3.1) of \cite{AGOS})
		\begin{align}
			\phi^*(z;\tau) &:= \sum_{k\ge0} (-1)^k e_k(\Lambda_\tau(0)) z^{2k+1} = \lim_{s\to0^+} \sum_{k\ge0} (-1)^k e_k(\Lambda_\tau(s)) z^{2k+1} \label{Fex}\\
			&= \lim_{s\to0^+} \sigma_s(z;\tau)= z\exp\left(-\widehat{G}_2(\tau)\frac{(2\pi iz)^2}2-\sum_{k\ge2}G_{2k}(\tau)\frac{(2\pi iz)^{2k}}{(2k)!}\right),\nonumber
		\end{align}
		where the last equality comes from \cite[equation~(3.3)]{AGOS}.
		Using equation (\ref{E:Trans}) gives that
		\begin{equation*}
			\phi^*\left(\frac z{c\tau+d};\frac{a\tau+b}{c\tau+d}\right) = (c\tau+d)^{-1} \phi^*(z;\tau).
		\end{equation*}
		Thus we have
		\begin{equation*}
			\sum_{k\ge1} (-1)^k e_k\left(\Lambda_{\frac{a\tau+b}{c\tau+d}}(0)\right) \left(\frac z{c\tau+d}\right)^{2k+1} = \frac1{c\tau+d} \sum_{k\ge1} (-1)^k e_k(\Lambda_\tau(0)) z^{2k+1}.
		\end{equation*}
		Comparing coefficients yields that
		\begin{equation*}
			e_k\lrb{\Lambda_{\frac{a\tau+b}{c\tau+d}}(0)} = (c\tau+d)^{2k} e_k\lrb{\Lambda_\tau(0)},
		\end{equation*}
		which is the claimed transformation law. The claim on the depth follows from part (2).

		\noindent(2) The fact that $e_0(\Lambda_\tau(0))=1$ is direct (for example by comparing coefficients in \eqref{Fex}). We next compute, using \eqref{Fex},
		\begin{align*}
			L(\phi^*(z;\tau)) ) &= L\left(z\exp\left(-\frac{\widehat{G}_2(\tau)}2(2\pi iz)^2-\sum_{k\ge2}\frac{G_{2k}(\tau)}{(2k)!}(2\pi iz)^{2k}\right)\right)\\
			&= \phi^*(z;\tau) \frac{-(2\pi iz)^2}2 L\lrb{\widehat{G}_2(\tau)} \hspace{-0.1cm} = -\frac{\pi z^2}2 \phi^*(z;\tau),
		\end{align*}
		using \eqref{Gh} and that we have, plugging into \eqref{Fex},
		\begin{equation}\label{eq:action-of-L-on-1/v}
			L\!\left(\frac1v\right) = -1.
		\end{equation}
		Thus we have, plugging into \eqref{Fex},
		\begin{equation*}
			\sum_{k\ge0} (-1)^k L\left(e_k(\Lambda_\tau(0))\right) z^{2k+1} = -\frac{\pi}2\sum_{k\ge0} (-1)^k e_k(\Lambda_\tau(0)) z^{2k+3}.
		\end{equation*}
		Comparing coefficients gives
		\begin{equation*}
			L(e_k(\Lambda_\tau(0))) = \frac{\pi}2e_{k-1}(\Lambda_\tau(0)),
		\end{equation*}
		as claimed.

		\noindent(3) We use \eqref{Fex}, \eqref{eq:action-of-L-on-1/v}, and then \cite[equation (7)]{Z1991} to write
		\begin{align}
			\phi(z;\tau) &= e^{-\frac{\pi z^2}{2v}} z\exp\left(-\widehat{G}_2(\tau)\frac{(2\pi iz)^2}2-\sum_{k\ge2}G_{2k}(\tau)\frac{(2\pi iz)^{2k}}{(2k)!}\right) \nonumber\\
			&= z\exp\left(-G_2(\tau)\frac{(2\pi iz)^2}2-\sum_{k\ge2}G_{2k}(\tau)\frac{(2\pi iz)^{2k}}{(2k)!}\right) = \frac{\vartheta(z;\tau)}{-2\pi\eta^3(\tau)}.\label{eq:theta}
		\end{align}
		From equation \eqref{E:TEl} we obtain, for $\lambda,\mu\in\Z$,
		\begin{align*}
			\phi(z+\lambda\tau+\mu;\tau) = (-1)^{\lambda+\mu} q^{-\frac{\lambda^2}2} \zeta^{-\lambda} \phi(z;\tau).
		\end{align*}
		Next, using \eqref{E:TMod} and \eqref{eq:eta} gives, for $
		\begin{psmallmatrix}
			a&b\\
			c&d
		\end{psmallmatrix}
		\in\SL_2(\Z)$, we have
		\begin{align*}
			\phi\left(\frac z{c\tau+b};\frac{a\tau+b}{c\tau+d}\right)
			&= (c\tau+d)^{-1} e^\frac{\pi icz^2}{c\tau+d} \phi(z;\tau).\qedhere
		\end{align*}
	\end{proof}
	\section{Proof of Theorem \ref{Jacobi}}\label{sec-thm-3}
	In this section we prove Theorem \ref{Jacobi}.

	\begin{proof}[Proof of Theorem \ref{Jacobi}]\hspace{0cm}\newline
	\noindent (1) Let $\phi(z;\t)$ be a meromorphic Jacobi form of weight $k$, index $m$, and torsion divisor $D$, and let $a$ be the order of $\phi(z;\t)$ at $z=0$. Then by \Cref{second}, there exist a weight $k+a$ modular form $f_\phi$ and Eisenstein-theta traces $\operatorname{Tr}_{n}(D,\psi_J;\t)$ such that
	\begin{align*}
		\phi(z;\tau) &= f_\phi(\tau) \sum_{n\ge0} \operatorname{Tr}_n (D,\psi_J;\tau) (2\pi iz)^{n+a}
	\end{align*}
	in a neighbourhood of $z=0$.
	We let
	\begin{equation}\label{PhisPhi}
		\phi^*(z;\tau) := e^\frac{\pi mz^2}v \phi(z;\tau).
	\end{equation}
	Then, in a neighbourhood of $0$, we have
	\begin{align}
		\phi^*(z;\tau) &= f_\phi(\tau) \sum_{j\ge0} \frac{\left(\frac{\pi mz^2}v\right)^j}{j!} \sum_{\ell\ge0} \operatorname{Tr}_\ell (D,\psi_J;\tau) (2\pi iz)^{\ell+a} \nonumber\\
		&= f_\phi(\tau) \sum_{n\ge0} z^{n+a} \sum_{\substack{j,\ell\ge0\\n=2j+\ell}} \frac{\left(\frac{\pi m}v\right)^j}{j!} (2\pi i)^{\ell+a}\operatorname{Tr}_\ell (D,\psi_J;\tau)\nonumber\\
		&= f_\phi(\tau) \sum_{n\ge0} c_{n}(\tau) z^{n+a},\label{dec}
	\end{align}
	where
	\begin{equation}\label{eq:c-tr}
		c_{n}(\tau) := \sum_{0\le j\le\frac{n}2} \frac{\left(\frac{\pi m}v\right)^j}{j!} (2\pi i)^{n-2j+a} \operatorname{Tr}_{n-2j} (D,\psi_J;\tau)= \widehat{\mathrm{Tr}}_{n}(D,\psi_J;\t),
	\end{equation}
	by definition \eqref{eq:Tr-hat}.

	We next prove modular transformation properties of $\widehat{\mathrm{Tr}}_{n}(D,\psi_J;\tau)$. It can be shown that $\phi^*(z;\tau)$ satisfies the modular transformation formula of a weight $k$ and index $0$ Jacobi form (see the proof of \cite[Proposition~3.1]{B2018}).
	We also have that $f_\phi$ is a modular form of weight $k+a$. Thus we obtain from \eqref{dec}
	\begin{align*}
		\phi^*\left(\frac z{c\tau+d};\frac{a\tau+b}{c\tau+d}\right) 
		&= f_\phi\left(\frac{a\tau+b}{c\tau+d}\right) \sum_{n\ge0} c_n\left(\frac{a\tau+b}{c\tau+d}\right) \left(\frac z{c\tau+d}\right)^{n+a} \\
		&=(c\tau+d)^{k+a} f_\phi(\tau)\sum_{n\ge0} c_n\left(\frac{a\tau+b}{c\tau+d}\right) \left(\frac z{c\tau+d}\right)^{n+a}\\
		&= (c\tau+d)^k \phi^*(z;\tau)= (c\tau+d)^k f_\phi(\tau) \sum_{n\ge0} c_n(\tau) z^{n+a}.
	\end{align*}
	Thus
	\begin{equation*}
		c_n\left(\frac{a\tau+b}{c\tau+d}\right) = (c\tau+d)^{n} c_n(\tau).
	\end{equation*}
	Using \eqref{eq:c-tr} this gives
	\begin{equation*}
		\widehat{\mathrm{Tr}}_{n}\left(D,\psi_J;\frac{a\tau+b}{c\tau+d}\right)=c_n\left(\frac{a\tau+b}{c\tau+d}\right) = (c\tau+d)^{n} c_n(\tau) = (c\tau+d)^{n}\widehat{\mathrm{Tr}}_{n}\left(D,\psi_J;\t\right).
	\end{equation*}
	Since by definition $\widehat{\mathrm{Tr}}_{n}$, is also a polynomial in $\frac{1}{v}$ with holomorphic coefficients and degree at most $\lfloor\frac{n}{2}\rfloor$, $\widehat{\mathrm{Tr}}_{n}$ is an almost holomorphic modular form of weight $n$ and depth at most $\lfloor\frac{n}{2}\rfloor$. By \eqref{eq:Tr-hat}, we have that
	\begin{align*}
		\lim_{\overline{\t}\to-i\infty} \widehat{\mathrm{Tr}}_{n}(D,\psi_J;\t) = (2\pi i)^{n+a} \mathrm{Tr}_{n}(D,\psi_J;\t).
	\end{align*}
	This proves (1).

	\noindent (2) Using \eqref{eq:c-tr} and the definition of $\operatorname{Tr}_n(D,\psi_J;\t)$ for $n=0,1$, we get
	\[
		\widehat{\mathrm{Tr}}_{0}\left(D,\psi_J;\t\right) = (2\pi i)^a, \quad \widehat{\mathrm{Tr}}_1(D,\psi_J;\t) = -(2\pi i)^{1+a}G_{1, D}(\t).
	\]
	We next compute the action of the lowering operator on $\phi^*$.
	We use \eqref{dec}, \eqref{PhisPhi}, and \eqref{eq:action-of-L-on-1/v} to get
	\begin{align*}
		f_\phi(\tau)\sum_{n\ge0} L\left(c_n(\tau)\right)z^{n+a+2}&=L\left(\phi^*(z;\tau)\right) =-\pi m\phi^*(z;\tau) z^2 \\
		&= -\pi mf_\phi(\tau) \sum_{n\ge0} c_n(\tau) z^{n+a+2}.
	\end{align*}
	Thus 
	\begin{equation*}
		L\left(c_n(\tau)\right)=-\pi mc_{n-2}(\tau).
	\end{equation*}
	Thus for $n\ge2$, we have, using \eqref{eq:c-tr},
	\begin{equation*}
		L\lrb{\widehat{\mathrm{Tr}}_{n}\left(D,\psi_J;\t\right)}=L(c_n(\tau)) = -\pi mc_{n-2}(\tau) = -\pi m \widehat{\mathrm{Tr}}_{n-2}\left(D,\psi_J;\t\right).
	\end{equation*}
	This completes the proof of (2).
	\end{proof}
	\section{Proof of \Cref{Ramanujan-U-V}}\label{sec-thm-4}
	We next construct almost holomorphic modular forms from $U_{2n}$ and $V_{2n}$.
	\begin{proof}[Proof of \Cref{Ramanujan-U-V}]
		We first prove  the results on $U_{2n}$.
		We use \eqref{eq:U-V-Tr} to write
		\begin{align}
				U(z;\t)&:=\sum_{n\ge 0} (-1)^nU_{2n}(q) \frac{(\pi z)^{2n+1}}{(2n+1)!}\nonumber\\
				&= \pi z\sum_{n\ge 0} \sum_{\lambda\vdash n} \prod_{k=1}^n\frac{1}{m_k!}\lrb{\frac{B_{2k}E_{2k}(\t)}{2k(2k)!}}^{m_k} (2\pi i z)^{2n}.\label{eq:U-as-cycleindex}
		\end{align}
		If we take $x_k=\frac{B_{k}E_k(\t)}{k!}$ and $w=2\pi iz$ in Lemma~\ref{lem:cycle-index}, then we get
	\begin{multline}\label{eq:cycle-index-as-exp}
		\sum_{n\ge 0} \sum_{\substack{\lambda\vdash n\\ \lambda=\lrb{1^{m_1},2^{m_2},\ldots,n^{m_n}}}}  \prod_{k=1}^n\frac{1}{m_k!}\lrb{\frac{B_{k} E_{k}(\t)}{k\cdot k!}}^{m_k} (2 \pi iz)^{n}\\
			 = \exp\left(\sum_{n\ge 1} \frac{B_{2n}E_{2n}(\t)}{2n}\frac{(2\pi iz)^{2n}}{(2n)!}\right),
	\end{multline}
	where we use that $E_{2n+1}(\t)=0$. We next also use this fact for the left-hand side. Note that we have that $\prod_{k=1}^n\frac{1}{m_k!}(\frac{B_{k} E_{k}(\t)}{k\cdot k!})^{m_k}=0$ if $m_k\neq0$ for some odd $k$ in the product, and in particular if $n$ is odd. Hence we can write the left-hand side as
		\begin{align*}
			\sum_{n\ge 0} \sum_{\substack{\lambda\vdash 2n\\ \lambda=\lrb{1^0,2^{m_2},3^0,4^{m_4},\ldots,(2n)^{m_{2n}}}}}\prod_{k=1}^{n} \frac1{m_{2k}!}\left(\frac{B_{2k}E_{2k}(\tau)}{2k(2k)!}\right)^{m_{2k}}(2\pi iz)^{2n} .
		\end{align*}
			We have that $\lambda=\lrb{1^0,2^{m_2},3^0,4^{m_4},\ldots,(2n)^{m_{2n}}}\vdash 2n$ is in one to one correspondence with $(1^{\mu_1},2^{\mu_2},\ldots,n^{\mu_n})\vdash n$ with $\mu_j=m_{2j}$, hence the above can be written as
		\begin{align}\label{eq:cycle-index}
			\sum_{n\ge 0} \sum_{\substack{\lambda\vdash n\\ \lambda=\lrb{1^{m_1},2^{m_2},\ldots,n^{m_n}}}}\prod_{k=1}^{n} \frac1{m_{k}!}\left(\frac{B_{2k}E_{2k}(\tau)}{2k(2k)!}\right)^{m_{k}}(2\pi iz)^{2n}.
		\end{align}
		Combining \eqref{eq:U-as-cycleindex}, \eqref{eq:cycle-index-as-exp}, and \eqref{eq:cycle-index} gives that
		\begin{align*}
				U(z;\t)=\pi z\exp\left(\sum_{n\ge 1} \frac{B_{2n}E_{2n}(\t)}{2n}\frac{(2\pi iz)^{2n}}{(2n)!}\right).
			\end{align*}
			Using \eqref{eq:theta}, this equals
			\begin{align*}
				\frac{\vartheta\lrb{z;\tau}}{-2\eta^3(\tau)},
			\end{align*}
			which is shown in the proof of \Cref{main} (3) to be a weight $-1$ and index $\frac{1}{2}$ Jacobi form. Next we prove that
			\begin{align}\label{eq:U-as-trace}
				U(z;\t) = \frac{1}{2i} \sum_{n\ge 0}\mathrm{Tr}_n([0],\psi_J;\t)(2\pi i z)^{n+1}.
			\end{align}
			Recall that we have $G_{2k,0}(\t)=G_{2k}(\t)$ and $G_{2k+1,0}(\t)=0$ by definition. By \eqref{eq:trace} we have
		\begin{align*}
			\mathrm{Tr}_{2n}([0],\psi_J;\t) &= \sum_{\lambda=(1^{m_1},2^{m_2},\ldots,(2n)^{m_{2n}})\vdash 2n} \psi_J(\lambda) \mathcal{G}_{[0],\lambda}(\t)\\
			&= \sum_{\lambda=(1^{\mu_1},2^{\mu_2},\ldots,n^{\mu_{n}})\vdash n} \frac{(-G_{2}(\t))^{\mu_1}(-G_4(\t))^{\mu_1}\cdots (-G_{2n}(\t))^{\mu_{n}}}{\prod_{j=1}^n \mu_{j}!(2j)!^{\mu_{j}}},
		\end{align*}
		where we again use that $\lambda=\lrb{1^0,2^{m_2},3^0,4^{m_4},\ldots,(2n)^{m_{2n}}}\vdash 2n$ is in one to one correspondence with $(1^{\mu_1},2^{\mu_2},\ldots,n^{\mu_n})\vdash n$ with $\mu_j=m_{2j}$.
		We use \eqref{eq:Ek} to write this as
		\begin{align*}
			\sum_{\lambda\vdash n} \prod_{j=1}^n \frac{( B_{2j}E_{2j}(\t))^{m_j}}{(2j)^{m_j} m_j!(2j)!^{m_j}}=\frac{\mathrm{Tr}_{n}(\psi_1;\t)}{4^n(2n+1)!}=
			\frac{U_{2n}(q)}{4^n(2n+1)!},
		\end{align*}
		where the first equality comes from the definition of $\mathrm{Tr}_{n}(\psi_1,\t)$ and the second from \eqref{eq:U-V-Tr}.
		Plugging this back into the definition of $U$ in \eqref{eq:U-as-cycleindex} gives \eqref{eq:U-as-trace}.
			Hence we are in the setting of \Cref{second} with $D=[0]$, $a=1$, and $f_\phi=\frac{1}{2i}$.
			A direct calculation shows that
			\begin{align*}
				\widehat{U}_{2n}(\tau) &= \widehat{\mathrm{Tr}}_{2n}([0],\psi_J;\t).
			\end{align*}
			By \Cref{Jacobi} (1), $\widehat{U}_{2n}$ is a weight $2n$ and depth at most $n$ almost holomorphic modular form satisfying
			\[
				\lim_{\overline{\t}\to-i\infty} \widehat U_{2n}(\tau) = 2i(-1)^n \frac{\pi^{2n+1}U_{2n}(q)}{(2n+1)!}.
			\]
			In this case the depth is exactly $n$ as by \eqref{eq:Uhat} the coefficient of $v^{-n}$ is $\frac{\pi^{n-1}iU_0(q)}{2^{n-1}n!} \neq 0$ since $U_0(q)=1$.
			This proves (1) and (3) for $U_{2n}$. By \Cref{Jacobi} (2), we have that $\widehat{U}_0=2\pi i$ and $L(\widehat{U}_{2n})=-\frac{\pi}{2}\widehat{U}_{2n-2}$ which proves (2) for $U_{2n}$.

			Next we show the results on $V_{2n}(q)$.
			We use \eqref{eq:U-V-Tr} to write
			\begin{align}
				V(z;\t):\!\!&=\sum_{n\ge 0} (-1)^nV_{2n}(q) \frac{(\pi z)^{2n}}{(2n)!}\label{eq:V-V-n}\\
				&= \sum_{n\ge 0} \sum_{\lambda\vdash n} \prod_{k=1}^n\frac{1}{m_k!}\lrb{\frac{\left(4^k-1\right)B_{2k} E_{2k}(\t)}{2k(2k)!}}^{m_k}  (2\pi iz)^{2n}\nonumber\\
				&=\exp\left(\sum_{n\ge 1}\left(4^n-1\right)\frac{B_{2n}E_{2n}(\t)}{2n}\frac{(2\pi iz)^{2n}}{(2n)!}\right) = \frac{\vartheta\lrb{2z;\t}}{2\vartheta\lrb{z;\t}},\nonumber
			\end{align}
			taking $x_k=\frac{(2^{k}-1)B_{k}E_k(\t)}{k!}$ and $w=2\pi iz$ in Lemma~\ref{lem:cycle-index} and applying a similar argument as for $U_{2n}$ and using \eqref{eq:theta}.

		Next we determine the transformation laws for $\frac{\vartheta\lrb{2z;\t}}{2\vartheta\lrb{z;\t}}.$ We start with the modular transformation. Using \eqref{E:TMod}, we have for $\left(\begin{smallmatrix}a&b\\c&d\end{smallmatrix}\right)\in\SL_2(\Z)$
		\begin{equation*}
			\frac{\vartheta\lrb{\frac{2z}{c\t+d};\frac{a\t+b}{c\t+d}}}{2\vartheta\lrb{\frac{z}{c\t+d};\frac{a\t+b}{c\t+d}}} = e^\frac{3\pi icz^2}{c\tau+d}  \frac{\vartheta(2z;\tau)}{2\vartheta(z;\tau)}.
		\end{equation*}
		Next we consider at the elliptic transformation. Using \eqref{E:TEl}, we have for $\lambda,\mu\in\Z$
		\begin{align*}
			\frac{\vartheta\lrb{2(z+\lambda\t+\mu);\t}}{2\vartheta\lrb{z+\lambda\t+\mu;\t}}
					&= (-1)^{\lambda+\mu} q^{-\frac{3\lambda^2}2} \zeta^{-3\lambda}\frac{\vartheta\lrb{2z;\t}}{2\vartheta\lrb{z;\t}}.
		\end{align*}
		So $\frac{\vartheta\lrb{2z;\t}}{2\vartheta\lrb{z;\t}}$, and hence $V(z;\t)$, is a Jacobi form of weight $0$ and index $\frac{3}{2}$. Using \eqref{eq:V-V-n}, we write
		\begin{align}
			V^*(z;\t)&:=e^{\frac{3\pi z^2}{2v}}V(z;\t) = \sum_{j\ge 0} \frac{\lrb{\frac{3\pi z^2}{2v}}^j}{j!} \sum_{n\ge 0} (-1)^nV_{2n}(q) \frac{(\pi z)^{2n}}{(2n)!}\nonumber\\[-0.3em]
			&= \sum_{n\ge 0} \widehat{V}_{2n}(\t) z^{2n}, \label{eq:V-hat}
		\end{align}
		where the last equality follows from definition \eqref{eq:V}. It can be shown that $V^*(z;\t)$ satisfies the modular transforms like  a weight $0$ and index $0$ Jacobi form (see the proof of \cite[Proposition~3.1]{B2018}).
		Using \eqref{eq:V-hat} with this gives
		\begin{align*}
			\sum_{n\ge 0} \widehat{V}_{2n}\lrb{\frac{a\tau+b}{c\tau+d}} \lrb{\frac{z}{c\tau + d}}^{2n}&=V^*\lrb{\frac{z}{c\tau + d};\frac{a\tau+b}{c\tau+d}}=V^*(z;\t)\\&=\sum_{n\ge 0} \widehat{V}_{2n}(\t) z^{2n}.
		\end{align*}
		Comparing the $(2n)$-th coefficient gives that 
		\begin{align*}
			\widehat{V}_{2n}\lrb{\frac{a\tau+b}{c\tau+d}} = (c\t+d)^{2n} \widehat{V}_{2n}(\tau).
		\end{align*}
		Also by definition $\widehat{V}_{2n}(\t)$ is a polynomial in $\frac{1}{v}$ with holomorphic coefficients and degree $n$. This is true because the term corresponding to $v^{-n}$ is $\frac{3^n \pi^n V_0(q)}{2^nn!} \neq0$ since $V_0(q)=1$. Hence $\widehat{V}_{2n}$ is an almost holomorphic modular form of weight $2n$ and depth $n$. This proves part (1) for $V$. To prove part (2) for $V$, first note that by definition $\widehat{V}_0=1$. We use \eqref{eq:action-of-L-on-1/v} to get
		\[
			L\lrb{V^*(z;\t)}=L\lrb{e^{\frac{3\pi z^2}{2v}}V(z;\t)} = -\frac{3\pi}{2}z^2 e^{\frac{3\pi z^2}{2v}}V(z;\t)= -\frac{3\pi}{2}z^2 V^*(z;\t).
		\]
		Combining this with \eqref{eq:V-hat} then gives
		\begin{align*}
			\sum_{n\ge 0} L\lrb{\widehat{V}_{2n}(\t)} z^{2n} = -\frac{3\pi}{2} \sum_{n\ge 0} \widehat{V}_{2n}(\t) z^{2n+2}.
		\end{align*}
		Comparing the  $(2n)$-th coefficient gives part (2) for $V$.
	Part (3) for $V$ follows by taking the limit $\overline{\t}\to i\infty$ in the definition of $\widehat{V}_{2n}$.
	\end{proof}
	
\end{document}